\documentclass{amsproc}

\usepackage{amsmath}
\usepackage{amssymb}
\usepackage{rotating}
\usepackage{graphicx}
\usepackage[tableposition=below]{caption}
\usepackage{bm}

\DeclareGraphicsExtensions{.png,.pdf,.eps}
\usepackage[all,2cell,cmtip]{xy}
\usepackage{tikz}
\usepackage{color}
\usepackage{longtable}
\usepackage{soul}
\newtheorem{theorem}{Theorem}[section]
\newtheorem{lemma}[theorem]{Lemma}
\newtheorem{corollary}[theorem]{Corollary}
\newtheorem{proposition}[theorem]{Proposition}

\theoremstyle{definition}

\newtheorem{definition}[theorem]{Definition}
\newtheorem{notation}[theorem]{Notation}

\newtheorem{remark}[theorem]{Remark}

\newtheorem{conj}[theorem]{Conjecture}
\newtheorem{NC}[theorem]{Notation and Convention}
\newtheorem{NA}[theorem]{Notation and Assumptions}

\def\C{{\mathbb C}}

\def\P{{\mathbb P}}

\def\R{{\mathbb R}}
\def\Z{{\mathbb Z}}

\def\a{{\bf a}}

\def\cE{{\mathcal E}}

\def\cG{{\mathcal{G}}}

\def\cM{{\mathcal M}}

\def\cO{{\mathcal{O}}}

\def\cS{{\mathcal S}}

\def\cU{{\mathcal U}}

\def\cW{{\mathcal W}}

\def\cOperatorname#1{\mathop{\rm #1}\nolimits}

\def\codim{\cOperatorname{codim}}

\def\deg{\cOperatorname{deg}}

\def\ME{{\cOperatorname{ME}}}

\newcommand{\cME}[1]{\cOverline{\ME}}

\makeindex

\begin{document}

\title{Fano manifolds of coindex three admitting nef tangent bundle}

\author{Kiwamu Watanabe}
\date{\today}
\address{Course of Matheorematics, Programs in Matheorematics, Electronics and Informatics, 
Graduate School of Science and Engineering, Saitama University.
Shimo-Okubo 255, Sakura-ku Saitama-shi, 338-8570, Japan.}
\email{kwatanab@rimath.saitama-u.ac.jp}
\thanks{The author is partially supported by JSPS KAKENHI Grant Number 17K14153.}

\subjclass[2010]{14J40, 14J45, 14M17.}
\keywords{}

\begin{abstract} We prove that any Fano manifold of coindex three admitting nef tangent bundle is homogeneous. 
\end{abstract}

\maketitle

\section{Introduction} 

By Mori's solution of the Hartshorne conjecture \cite{Mori1}, the only projective manifold with ample tangent bundle is the projective space. An analytic counterpart of the Hartshorne conjecture is the Frankel conjecture: the only compact K\"alher manifold with positive holomorphic bisectional curvature is the projective space, which can be obtained as a corollary of the Hartshorne conjecture. An analytic proof of the Frankel conjecture was obtained by Y. T. Siu and S. T. Yau \cite{SiuYau}. Following these works, N. Mok classified compact K\"ahler manifold with nonnegative holomorphic bisectional curvature \cite{Mk0}. Based on Mok's result, F. Campana and T. Peternell \cite{CP1} conjectured that any compact K\"ahler manifold with nef tangent bundle admits a finite \'etale cover $\tilde{X} \to X$ such that the Albanese map $\tilde{X} \to {\rm Alb}(\tilde{X})$ is a fiber bundle whose fibers are rational homogeneous manifolds. 
By the work of J. P. Demailly, Peternell and M. Schneider \cite{DPS}, the conjecture can be reduced to the Fano case: 
\begin{conj}[Campana-Peternell Conjecture]\label{conj:CP} Any complex Fano manifold with nef tangent bundle is homogeneous.
\end{conj}
This conjecture is called the {\it Campana-Peternell conjecture}. The conjecture was proved in dimension at most three by Campana and Peternell \cite{CP1}, in dimension four by Campana and Peternell \cite{CP2}, Mok \cite{Mk} and J. M. Hwang \cite{Hw4}, and in dimension five by the author \cite{Wa2} and A. Kanemitsu \cite{Kane}. It is also known that the conjecture holds for complete intersections of hypersurfaces by R. Pandharipande \cite{Pan} 
and horospherical varieties by Q. Li \cite{Li2}. However the conjecture is widely open in general. For the recent development of the topic, we refer the reader to \cite{MOSWW}.
 
For a Fano manifold $X$, the {\it index} $i_X$ of $X$ is the maximal integer dividing the anticanonical divisor $-K_X$ in ${\rm Pic}(X)$, while the {\it coindex} of $X$ is defined as $\dim X +1 -i_X$. According to the classification of Fano manifolds with small coindex \cite{KO, Fuj1, Fuj2}, it is easy to check that the Campana-Peternell conjecture is true for Fano manifolds of coindex at most two (see for instance \cite{KO} and \cite[Corollary~5.7]{SW}). The purpose of this paper is to study the conjecture for Fano manifolds of coindex three:

\begin{theorem}\label{MT} 
Let $X$ be a Fano manifold of coindex three. If $X$ has a nef tangent bundle, then $X$ is isomorphic to one of the following:
\begin{enumerate}
\item the $10$-dimensional spinor variety $S_4$;
\item the $8$-dimensional Grassmannian $G(2, 6)$; 
\item the $7$-dimensional symplectic Grassmannian $SG(2, 5)$;
\item the $6$-dimensional symplectic Grassmannian $SG(3, 6)$;
\item $\P^3 \times \P^3$;
\item $\P^2 \times Q^3$; or 
\item $\P(T_{\P^3})$.
\end{enumerate}
In particular, $X$ is homogeneous.
\end{theorem}

This paper is organized as follows. In Section~2, we review the classification of Fano manifolds of coindex three and narrow down the possibilities of Fano manifolds of coindex three admitting nef tangent bundle. Consequently, we shall see that it is enough to consider the two cases. The first case is the case when the fundamental linear system $|-\frac{1}{n-2}K_X|$ determines a morphism which is not an embedding. In this case, since $X$ can be described as a weighted complete intersection, we can compute the degree of the top Chern class of $X$. By using the description, we conclude that the tangent bundle of $X$ is not nef via intersection theory of higher-codimensional cycles. This shall be done in Section~3. The second case is the case when $X$ is a smooth linear section of the $10$-dimensional spinor variety $S_4\subset \P^{15}$. Applying previous results, it is sufficient to study the cases $\dim X=8$ and $\dim X=7$. When $\dim X=8$, we prove the existence of a line whose normal bundle in $X$ is not nef.  When $\dim X=7$, we prove the non-nefness of the tangent bundle via the theory of varieties of minimal rational tangents. These shall be done in Section~4. 

\subsection{Notation and Conventions}\label{subsec:NC} Throughout this paper, we work over the complex number field $\C$. We will use the basic notation and definitions in \cite{Har}. 
\begin{itemize}
\item A smooth projective variety is called a {\it projective manifold}.
\item A $k$-dimensional linear subspace of $\P^n$ is called a {\it $k$-plane}.
\item For a projective manifold $X$, we denote the Picard number of $X$ by $\rho_X$.
\item For a projective manifold $F$, an {\it $F$-bundle} means a smooth morphism $f:Y \to X$ between projective manifolds whose fibers are isomorphic to $F$.
\item For a vector bundle $\cE$, $\cE^{\ast}$ denotes its dual vector bundle.
\item For a vector bundle $\cE$, $\P(\cE)$ denotes its projectivization ${\rm Proj}({\rm Sym}\,\cE)$.
\item $Q^n$ denotes a smooth quadric hypersurface in $\P^{n+1}$.
\item $G(k, n)$ denotes the Grassmannian parametrizing $k$-dimensional linear subspaces in $\C^n$.
\item $SG(k, 2n)$ denotes the symplectic Grassmannian parametrizing isotropic $k$-dimensional linear subspaces in $\C^{2n}$ with respect to a fixed symplectic form on $\C^{2n}$.
\item $OG(k, n)$ denotes the orthogonal Grassmannian parametrizing isotropic $k$-dimensional linear subspaces in $\C^{n}$ with respect to a fixed nondegenerate symmetric bilinear form on $\C^{n}$.
\item $S_4 \subset \P^{15}$ denotes the $10$-dimensional spinor variety, which is a connected component of the orthogonal Grassmannian $OG(5, 10)$. The embedding $S_4  \subset \P^{15}$ is given by a half-spinor representation. 
\item $\cS$ denotes the spinor bundle on $Q^{2k+1}$. $\cS'$ and $\cS''$ denote the spinor bundles on $Q^{2k}$. For the definition of spinor bundles, see \cite[Definition~1.3]{Ott}.
\item $V_5$ denotes a $3$-dimensional smooth linear section of $G(2,5) \subset \P^9$ in the Pl\"ucker embedding. By \cite[Theorem~7.6]{Fuj2}, all smooth linear sections of $G(2,5) \subset \P^9$ are projectively equivalent.
\end{itemize}

\section{Preliminaries}

In this section, firstly we review the classification of Fano manifolds of coindex three. By using the classification, secondly we narrow down the possibilities of Fano manifolds of coindex three admitting nef tangent bundle.

A {\it Fano manifold} means a smooth projective variety with ample anticanonical divisor $-K_X$.   
Let $X$ be an $n$-dimensional Fano manifold $X$ of coindex $3$ and $\rho_X=1$. For an ample generator $H$ of ${\rm Pic}(X)$, the self-intersection number $H^n$ is called the {\it degree} of $X$ and the integer $g=g(X):=\dfrac{H^n}{2}+1$ is called the {\it genus} of $X$. 
Let us recall the classification of Fano manifolds of coindex $3$ due to S. Mukai:

\begin{theorem}[{\cite{Mu, Mella}, cf. \cite[Chapter~5]{Isk}}]\label{thm:Mukai} Let $X$ be a Fano manifold of coindex three. If $n:=\dim X \geq 6$, then $X$ is isomorphic to $\P^3 \times \P^3$ or we have $\rho_X=1$. If $\rho_X=1$, then the genus $g$ of $X$ satisfies with $2 \leq g \leq 9$. Furthermore $X$ can be described as follows:
\begin{enumerate} 
\item If $g=2$, then $X$ is a weighted hypersurface of degree $6$ in the weighted projective space $\P(3, 1, \ldots, 1)$. In this case, $X$ is a double cover $\varphi: X \to \P^n$ branched along a sextic.
\item If $g=3$, then $X$ is one of the following:
\begin{enumerate} 
\item $X$ is a weighted complete intersection of type $(4,2)$ in the weighted projective space $\P(2, 1, \ldots, 1)$. In this case, $X$ is a double cover $\varphi: X \to Q^n \subset \P^{n+1}$ branched along the intersection of $Q$ with a quartic hypersurface. 
\item $X \subset \P^{n+1}$ is a quartic hypersurface.
\end{enumerate}
\item If $g=4$, then $X \subset \P^{n+2}$ is a complete intersection of type $(2,3)$.
\item If $g=5$, then $X \subset \P^{n+3}$ is a complete intersection of type $(2,2,2)$.
\item If $g \geq 6$, then $X$ is a linear section of a variety
$$\Sigma^{n(g)}_{2g-2} \subset \P^{g+n(g)-2}
$$ of dimension $n(g)$ and degree $2g-2$, which is one of the following:
\begin{enumerate} 
\item[{($g=6$)}]  $\Sigma^6_{10}\subset \P^{10}$ is a quadric section of the cone over the Grassmannian $G(2, 5)\subset \P^9$ in the Pl\"ucker embedding;
\item[{($g=7$)}] $\Sigma^{10}_{12}= S_4  \subset \P^{15}$ is the $10$-dimensional spinor variety in the embedding induced by the half-spinor representation;
\item[{($g=8$)}] $\Sigma^{8}_{14}= { G}(2,6)  \subset \P^{14}$ is the Grassmannian $G(2, 6)$ in the Pl\"ucker embedding;
\item[{($g=9$)}] $\Sigma^{6}_{16}= SG(3, 6)  \subset \P^{13}$ is the symplectic Grassmannian $SG(3, 6)$ in the Pl\"ucker embedding.
\end{enumerate}
\end{enumerate}
\end{theorem}

By using the above classification, let us start to study when Fano manifolds of coindex three admit nef tangent bundle. Firstly we define:

\begin{definition}[{\cite[Definition 1.4]{MOSWW}}] A Fano manifold $X$ with nef tangent bundle is called a {\it CP manifold}. 
\end{definition}

\begin{theorem}[{\cite{CP1, CP2, Hw4, Kane, Mk, Wa2, Pan}}]\label{thm:CP5}
The Campana-Peternell conjecture \ref{conj:CP} holds for CP manifolds of dimension at most five and complete intersections of hypersurfaces in a projective space.
\end{theorem}

Recall that, for a Fano manifold $X$, the {\it pseudoindex} $\iota_X$ of $X$ is the minimal anticanonical degree of rational curves on $X$. Thanks to results of Kanemitsu \cite{Kane3, Kane}, we can obtain the following theorem. All of the ideas of the proof are contained in \cite{Kane3, Kane}, but this result is not specified in these papers. For the reader's convenience, we present a proof.

\begin{theorem}\label{them:6fold} Let $X$ be a CP $6$-fold of $\rho_X=1$ and $\iota_X=4$. Then $X$ is homogeneous.
\end{theorem}

\begin{proof}  Applying \cite[Proposition~4.1, Theorem~3.3]{Kane}, any CP $6$-fold of $\rho_X=1$ and $\iota_X=4$ satisfies one of the following;
\begin{enumerate}
\item $X$ is isomorphic to the symplectic Grassmannian $SG(3, 6)$ or 
\item $X$ admits the following commutative diagram
\[\xymatrix{
U \ar[r]^{f} \ar[d]^{\pi}  &  W \ar[d]^{q}  \ar[r]^{g} & X  \\
V \ar[r]^{p} & Y \\
} \]  with the properties: 
\begin{enumerate}
\item $f, g$ and $\pi$ are $\P^1$-bundles, and
\item $p$ and $q$ are smooth contractions of an extremal ray. 
\end{enumerate}
\end{enumerate}
Since the symplectic Grassmannian $SG(3, 6)$ is homogeneous, we only need to consider the second case. 
In this case, for any point $y \in Y$, let us denote by $V_y, W_y$ and $U_y$, respectively, the fibers of $y$ by $p, q$ and $p\circ \pi=q \circ f$. Since $\rho_X=1$ and any contraction of an extremal ray as in the above diagram is a smooth morphism, we have $\rho_{V_y}=1, \rho_{W_y}=1$ and $\rho_{U_y}=2$. This means that ${U_y}$ is a Fano manifold with two different $\P^1$-bundle structures. By \cite[Theorem~A.1]{MOSWW}, $U_y$ is isomorphic to $G/B$, where $G$ is a simple linear algebraic group of rank $2$ and $B$ is its Borel subgroup. It turns out that $W_y$ is $G/P$ for some parabolic subgroup $P$ of $G$. By the rigidity results on rational homogeneous manifolds \cite[Theorem~3.3.2]{Mok}, $q: W \to Y$ is a $G/P$-bundle. Then, applying \cite[Theorem~0.2]{Kane3} to $W$, we see that $W$ is homogeneous or it is isomorphic to the projectivized Ottaviani bundle $\P(\cG_o)$ in the sense of Kanemitsu \cite[Definition~2.1]{Kane3} (see Definition~\ref{def:ott} below). If $W$ is the projectivized Ottaviani bundle $\P(\cG_o)$, then $X$ is $Q^5$. This is a contradiction, because we assume $\iota_X=4$ but $\iota_{Q^5}=5$. So we may assume $W$ is homogeneous. Since the target of any contraction of any rational homogeneous manifold is again homogeneous, our assertion holds. 
\end{proof}

\begin{corollary}\label{cor:CP6} Let $X$ be a CP $6$-fold of coindex three. Then $X$ is homogeneous.
\end{corollary}

\begin{proof} By Theorem~\ref{thm:Mukai}, $X$ is isomorphic to $\P^3 \times \P^3$ or $\rho_X=1$. Since $\P^3 \times \P^3$ is homogeneous, we may assume $\rho_X=1$. In this case, it follows from the classification Theorem~\ref{thm:Mukai} that $X$ admits a rational curve $C$ on $X$ such that $-K_X\cdot C=4$. Hence the pseudoindex $\iota_X$ is equal to $4$. Then our claim follows from Theorem~\ref{them:6fold}. 
\end{proof}

\begin{proposition}[The case $g=8$]\label{prop:g=8} Let $X$ be a smooth linear section of the Grassmannian ${G}(2, 6)  \subset \P^{14}$ in the Pl\"ucker embedding. If the tangent bundle of $X$ is nef, then $X$ is isomorphic to the Grassmannian ${G}(2, 6)$ or the symplectic Grassmannian $SG(2, 6)$.
\end{proposition}

\begin{proof} According to Theorem~\ref{thm:CP5} and Corollary~\ref{cor:CP6}, we may assume $X$ is the Grassmannian ${G}(2, 6)$ or its smooth hyperplane section. Since any smooth hyperplane section of the Grassmannian ${G}(2, 6)  \subset \P^{14}$ is the symplectic Grassmannian $SG(2, 6)$ (see for instance \cite[Theorem~1.1~(1)]{saka}), our assertion holds.  
\end{proof}

\begin{remark}\label{rem:cases}
Combining Theorem~\ref{thm:CP5}, Corollary~\ref{cor:CP6} and Proposition~\ref{prop:g=8} with Theorem~\ref{thm:Mukai}, to prove Theorem~\ref{MT} it is enough to study the following two cases of $\rho_X=1$:
\begin{enumerate} 
\renewcommand{\labelenumi}{(\roman{enumi})}
\item the fundamental linear system $|-\frac{1}{n-2}K_X|$ determines a morphism which is not an embedding ($X$ is one of the varieties as in (1) and (2)-(a) of Theorem~\ref{thm:Mukai});
\item $g=7$ and $7 \leq \dim X \leq 9$.
\end{enumerate} 
\end{remark}

We shall deal with the case $\rm (i)$ in Section~\ref{sec:g=2,3} and the case $\rm (ii)$ in Section~\ref{sec:g=7}.  

\section{The cases where $g=2$ and $3$}\label{sec:g=2,3}

We study the cases (1) and (2)-(a) of Theorem~\ref{thm:Mukai} in a similar way to \cite[Section~4]{SW}. The following results were obtained during my joint work with T. Suzuki \cite{SW}. We begin with a quick review of some intersection theory of cycles.  

For the intersection theory, we refer the reader to \cite{Ful}. Let $X$ be a projective manifold of dimension $n$.  A {\it $k$-cycle} on $X$ means a finite formal linear combination $\sum a_i [Z_i]$ where $Z_i$ are $k$-dimensional closed subvarieties of $X$ and $a_i$ are real numbers. The cycle $\sum a_i [Z_i]$ is said to be {\it nef} if it has non-negative intersection numbers with all $k$-dimensional closed subvarieties.  We denote by $A_k(X)$ the group of $k$-cycles with real coefficients modulo rational equivalence on $X$. Then we have the degree homomorphism $\deg: {A}_0(X)\to \R$ (see \cite[Definition~1.4]{Ful}). 
Let $E$ be a vector bundle of rank $r$ on $X$. For each $k=0, 1, \ldots, r$, the {\it $k$-th Chern class} $c_k(E) \in A_k(X)$ is defined by the relations
$$\sum_{k=0}^r(-1)^k\pi^{\ast}c_k(E)\xi^{r-k}=0~\text{and}~c_0(E)=1,
$$
where $\xi$ is the divisor associated to the tautological line bundle $\cO_{\P(E)}(1)$ and $\pi: \P(E) \to X$ is the natural projection. We define the Chern class $c_r(X)$ of $X$ as the Chern class $c_r(T_X)$.

For a projective manifold $X$, the diagonal $\Delta_X \subset X \times X$ is said to be {\it nef} if $\Delta_X$ is nef as a cycle on $X \times X$. Remark that $\Delta_X$ is nef, provided that $T_X$ is nef. Thus, to prove $T_X$ is not nef for a given manifold $X$, it is sufficient to prove $\Delta_X$ is not nef. We review the following two results:

\begin{proposition}[{\cite[Proposition~8.1.12]{Ful}}]\label{prop:diag^2} For a projective manifold $X$, let $\Delta_X:=\{(x,x)\mid x \in X\} \subset X \times X$ be the diagonal of $X$. Then 
$$\deg \Delta_X^2= \deg c_n(X).
$$ 
\end{proposition}

\begin{proposition}[{\cite[Proposition~4.4]{LO}}]\label{prop:LO} Let $X$ be a projective manifold of dimension $n$ admitting a surjective generically finite morphism $f :X \to Y$  of degree $d$ to a projective manifold $Y$. Suppose that $\deg c_n(X)>d \deg c_n(Y)$. Then $\Delta_X$ is not nef.
\end{proposition}

By using the above results, we shall prove that the varieties as in (1) and (2)-(a) of Theorem~\ref{thm:Mukai} have non-nef tangent bundle. To prove it, we set up notation on weighted complete intersections. 
For a vector of positive integers $\a=(a_0, a_1, \ldots, a_m)$, let us denote by $\P(\a)$ the weighted projective space of type $\a$. 
Let $X \subset \P(\a)$ be an $n$-dimensional smooth weighted hypersurface of degree $d$. Assume that $X$ is contained in the smooth locus of $\P(\a)$ and $m \geq 4$. By \cite[Theorem~5.32]{KKA}, $\cO_X(1)$ generates ${\rm Pic}(X)$: ${\rm Pic}(X)=\Z[\cO_X(1)]$. We denote by $h \in A_{n-1}(X)$ the class corresponding to $\cO_X(1)$.

\begin{proposition}[{\cite[Proposition~4.4]{SW}}]\label{prop:weighted} Let $X \subset \P(\a)$ be a smooth weighted hypersurface of degree $d$. Assume that $X$ is contained in the smooth locus of $\P(\a)$ and $m \geq 4$. Then the degree of the top Chern class is given by 
$$\deg c_{m-1}(X)=\biggl\{ \sum_{i=0}^{n-1}e_{m-1-i}(a_0, \ldots,a_m)(-d)^i\biggr\} h^{m-1}.
$$
\end{proposition}

\begin{proposition}\label{prop:g=2} Let $X$ be a weighted hypersurface of degree $6$ in the weighted projective space $\P(3, 1, \ldots, 1)$. Then $\Delta_X$ is not nef. In particular, the tangent bundle of $X$ is not nef.  
\end{proposition}

\begin{proof} Applying \cite[Proposition~7]{DD}, the singular locus $\P_{\rm sing}$ of $\P(3, 1, \ldots, 1)$ consists of one point $\{(1: 0: \ldots: 0)\}$.
Since $\dim X \geq 3$, we have $$\codim_X(X \cap \P_{\rm sing}) \geq 3.$$ Hence $X$ is in general position relative to $\P_{\rm sing}$ in the sense of A. Dimca \cite[Definition~1]{Dim}. Then \cite[Proposition~8]{Dim} tells us that the singular locus of $X$ coincides with $X \cap \P_{\rm sing}$. Since $X$ is smooth by definition, $X$ is contained in the smooth locus of $\P(3, 1, \ldots, 1)$. 

Moreover, $X$ is a double cover branched along a sextic in $\P^n$. Applying Proposition~\ref{prop:weighted}, we obtain
$$\deg c_n(X) = 2\sum_{i=0}^ne_{n-i}(3,1^{n+1})(-6)^i.
$$
Remark that
$$e_k(3,1^{n+1})=3e_{k-1}(1^{n+1})+e_{k}(1^{n+1})=3\binom{n+1}{k-1}+\binom{n+1}{k}.
$$
Thus we have 
\begin{eqnarray}
 \deg c_n(X)&=&2\sum_{i=0}^n \biggl\{ 3\binom{n+1}{i+2}+\binom{n+1}{i+1} \biggr\}(-6)^i \nonumber \\
  &=&\dfrac{1}{6}\sum_{i=0}^{n-1} \binom{n+1}{i+2}(-6)^{i+2}-\dfrac{1}{3}\sum_{i=0}^{n}\binom{n+1}{i+1}(-6)^{i+1}  \nonumber \\
&=&\dfrac{1}{6}\biggl\{ (-5)^{n+1}+6(n+1)-1\biggr\} -\dfrac{1}{3}\biggl\{ (-5)^{n+1}-1\biggr\} \nonumber \\
&=&\dfrac{1}{6}\biggl\{ 6n+7-(-5)^{n+1}\biggr\}.  \nonumber 
\end{eqnarray}
If $n$ is odd, then $\deg c_n(X)$ is negative. Thus $\Delta_X$ is not nef. If $n$ is even, then it is straightforward to show the following:
$$\dfrac{1}{6}\biggl\{ 6n+7-(-5)^{n+1}\biggr\}>2(n+1)~\text{if}~n~\text{is even}.
$$ 
Hence it follows from Proposition~\ref{prop:LO} that $\Delta_X$ is not nef.
\end{proof}

\begin{proposition}\label{prop:g=3} Let $X$ be a weighted complete intersection of type $(4,2)$ in the weighted projective space $\P(2, 1, \ldots, 1)$. Then $\Delta_X$ is not nef. In particular, the tangent bundle of $X$ is not nef.  
\end{proposition} 

\begin{proof} By the same way as in Proposition~\ref{prop:weighted}, we obtain
$$\deg c_n(X)=\dfrac{(-3)^{n+2}+4n+7}{4}.
$$ From this description, we see that $\deg c_n(X)$ is negative if $n$ is odd. Thus Proposition~\ref{prop:diag^2} tells us that $\Delta_X$ is not nef. On the other hand, if $n$ is even, we have $\deg c_n(X)>4(n+1)$. By Proposition~\ref{prop:LO}, $\Delta_X$ is not nef.    
\end{proof}

\section{The cases where $g=7$}\label{sec:g=7}

\subsection{Ottaviani bundle}\label{subsec:ott}

Let us first describe some known results on spinor varieties and Ottaviani bundles without proof. All materials here are in \cite{Ott}. 
We start by considering $2k$-dimensional smooth quadric $Q^{2k}\subset \P^{2k+1}$. The linear spaces of maximal dimension of $Q^{2k}$ have dimension $k$. The Fano scheme of $k$-planes on $Q^{2k}$ consists of two irreducible families, each of which is a rational homogeneous variety of dimension $\dfrac{k(k+1)}{2}$, called the {\it spinor variety} $S_k$ of $Q^{2k}$. One often finds the literature that $S_{k}$ is denoted by $S_{k+1}$. In this paper, we employ the notation as in \cite{Ott}. We have 
\begin{itemize}
\item ${\rm Pic}(S_k) \cong \Z$; 
\item the ample generator $\cO_{S_k}(1)$ of ${\rm Pic}(S_k)$ is very ample; and  
\item $H^0(S_k, \cO_{S_k}(1))\cong \C^{2^k}$.
\end{itemize}
Let $Q^{2k+2} \subset \P^{2k+3}$ be a smooth quadric and $H$ a non-tangent hyperplane. Then $Q^{2k+1}=Q^{2k+2}\cap H$ is a smooth hyperplane section of $Q^{2k+2}$. Consider the natural morphism 
$$
\{\P^{k+1}\mid \P^{k+1} \subset Q^{2k+2}\} \to \{ \P^{k}\mid \P^{k} \subset Q^{2k+1}\}; \P^{k+1} \mapsto \P^{k+1} \cap H 
.$$  If we restrict this morphism to each component $S_{k+1}$ of $\{\P^{k+1}\mid \P^{k+1} \subset Q^{2k+2}\}$, we have an isomorphism 
$$ S_{k+1} \cong \{ \P^{k}\mid \P^{k} \subset Q^{2k+1}\}.
$$
Under this isomorphism, let us consider the following incidence correspondence: 
\begin{equation}\label{eq:diag}
\vcenter{
\xymatrix{
 &  F:=\{(x, \P^k)\mid x \in \P^k \subset Q^{2k+1}\}  \ar[dr]^{q} \ar[dl]_{p} &   \\
Q^{2k+1} & &S_{k+1} \\
} }
\end{equation}
For $x \in Q^{2k+1}$, consider an inclusion  
$$
i_x: (S_k)_x:=\{\P^k \mid x \in \P^k \subset Q^{2k+1}\} \hookrightarrow \{ \P^{k}\mid \P^{k} \subset Q^{2k+1}\}\cong S_{k+1}.
$$ 
Remark that $(S_k)_x \cong S_{k}$. By \cite[Corollary~1.2~(i)]{Ott}, we have 
\begin{itemize}
\item an isomorphism $i_x^{\ast}\cO_{S_{k+1}}(1) \cong \cO_{(S_k)_x}(1)$ and 
\item an injection $H^0((S_{k})_x, \cO(1))^{\ast} \hookrightarrow H^0(S_{k+1}, \cO(1))^{\ast}$. 
\end{itemize}
Then we have an embedding $\tau: Q^{2k+1} \hookrightarrow G(2^k, 2^{k+1})$. By definition, the spinor bundle $\cS$ on $Q^{2k+1}$ is $\tau^{\ast}\cU$, where $\cU$ is the universal bundle of $G(2^k, 2^{k+1})$. Then we see that the vector bundle $p_{\ast}q^{\ast}\cO_{S_{k+1}}(1)$ is the dual of the spinor bundle $\cS^{\ast}$ on $Q^{2k+1}$, and $p: F \to Q^{2k+1}$ is given by the bundle projection $F=\P_{Q^{2k+1}}(\cS^{\ast})\to Q^{2k+1}$. In particular, when $k=2$, we see that $S_3\cong Q^6$ by triality. Thus the above diagram  (\ref{eq:diag}) can be written as follows:
\begin{equation}\label{eq:diag:2}
\vcenter{
\xymatrix{
 &  F=\P_{Q^5}(\cS^{\ast})   \ar[dr]^{q} \ar[dl]_{p}&   \\
Q^{5} & &S_{3}\cong Q^6 \\
} }
\end{equation}

For any integer $0 \leq i \leq 5$, the Chow group $A_i(Q^5)_{\Z}$ with integral coefficients is isomorphic to $\Z$. We fix the positive generator $Z_i$ of $A_i(Q^5)_{\Z}$. Then, for simplicity, $Z \in A_i(Q^5)_{\Z}$ is denoted by $\alpha_i$ provided that $Z=\alpha_iZ \in A_i(Q^5)_{\Z}$. 
By \cite[Remark~2.9]{Ott}, the Chern classes of $\cS^{\ast}$ on $Q^5$ is $(c_1, c_2, c_3, c_4)=(2,2,2,0)$. In particular, generic sections of $\cS^{\ast}$ on $Q^5$ do not vanish anywhere. This tells us that $\cS^{\ast}$ contains a trivial subbundle  fitting into the following exact sequence of vector bundles:
$$
0 \to \cO_{Q^5} \to \cS^{\ast} \to \cG \to 0,
$$
where $\cG=\cS^{\ast}/ \cO_{Q^5}$ is a vector bundle of rank three on $Q^5$. The bundle $\cG$ is a stable vector bundle of rank three on $Q^5$ with Chern classes $(c_1, c_2, c_3)=(2,2,2)$. On the other hand, the converse also holds:

\begin{theorem}[{\cite[Section~3]{Ott}}]\label{them:ott} For a vector bundle $\cG$ of rank three on $Q^5$, the following are equivalent:
\begin{enumerate}
 \item $\cG$ is stable and its Chern classes are $(c_1, c_2, c_3)=(2,2,2)$;
 \item $\cG$ fits into the following exact sequence of vector bundles:
$$
0 \to \cO_{Q^5} \to \cS^{\ast} \to \cG \to 0.
$$
\end{enumerate}
\end{theorem}

\begin{definition}[{\cite[Definition~2.1]{Kane3}}]\label{def:ott} A vector bundle $\cG$ of rank three on $Q^5$ is said to be {\it Ottaviani bundle} if the equivalent conditions of Theorem~\ref{them:ott} hold for $\cG$. In the following, any Ottaviani bundle is denoted by $\cG_o$.
\end{definition}

\begin{proposition}\label{prop:ott} Under the notation as in the above diagram {\rm(\ref{eq:diag:2})}, let $Q^5$ be a smooth hyperplane section of $Q^6 \cong S_3$. Then the preimage $q^{-1}(Q^5)\to Q^5$ is given by $\P_{Q^5}(\cG_o)\to Q^5$, where $\cG_o$ is an Ottaviani bundle.
\end{proposition}

\begin{proof} We recall some arguments in the proof of \cite[Theorem~3.5]{Ott}. Since $Q^6\cong S_3$ parametrizes $2$-planes in $Q^5$, we have a closed immersion $Q^6 \hookrightarrow G(3, 7)$. Then the pullback $\cU^{\ast}$ of the dual of the universal bundle by the above immersion $Q^6 \hookrightarrow G(3, 7)$ coincides with $q_{\ast}p^{\ast}\cO_{Q^5}(1)$. 
Then we have an exact sequence of vector bundles on $Q^6$: 
$$
0 \to \cO_{Q^6} \to \cS'^{\ast} \to \cU^{\ast} \to 0,
$$ 
where $\cS'$ is a spinor bundle on $Q^6$.
By \cite[Theorem~1.4]{Ott}, we have $\cS'^{\ast}|_{Q^5}\cong \cS$. Restricting this sequence to a hyperplane section $Q^5$ of $Q^6$, we obtain an exact sequence:
$$
0 \to \cO_{Q^5} \to \cS^{\ast} \to \cU^{\ast}|_{Q^5} \to 0.
$$ 
From this sequence, $\cU^{\ast}|_{Q^5}$ is nothing but an Ottaviani bundle on $Q^5$.
Since $q: F \to Q^6$ is given by $\P_{Q_6}(\cU^{\ast}) \to Q^6$, our assertion holds. 
\end{proof}

\subsection{Linear sections of $10$-dimensional spinor variety}\label{subsec:spinor}

Let $X$ be a smooth linear section of the $10$-dimensional spinor variety $S_4  \subset \P^{15}$. We shall prove the only smooth linear section of $S_4$ with nef tangent bundle is $S_4$ itself. To prove it, we may assume that $7 \leq \dim X \leq 9$ by Remark~\ref{rem:cases}. Moreover when $\dim X =9$, $X$ is a horospherical variety. This case was done by \cite[Proposition~4.6]{Li2}. Thus we only need to deal with the cases of $\dim X=8$ and $\dim X=7$.

\subsubsection{\rm \bf The case of $\dim X=8$}

Let us begin with a review of A. G. Kuznetsov's results \cite[Section~6]{Kuz}. Throughout this subsection, let $X$ be a smooth linear section $X \subset S_4$ of codimension two. 

\begin{definition}[{\cite[Definition~6.2]{Kuz}}]\label{def:special} A smooth linear section $X \subset S_4$ of codimension two is said to be {\it special} (resp. {\it non-special}), if $X$ contains (resp. does not contain) a $4$-plane.
\end{definition}

\begin{proposition}[{\cite[{Lemma~6.4}]{Kuz}}]\label{prop:line} A smooth linear section $X \subset S_4$ of codimension two is special if and only if there is a line $\ell \subset X$ such that
$$
N_{\ell/X}\cong \cO_{\ell}(-2) \oplus \cO_{\ell}^{\oplus 6}.
$$
\end{proposition}

By Proposition~\ref{prop:line}, a smooth linear section $X \subset S_4$ of codimension two has non-nef tangent bundle if $X$ is special. So we may assume that $X$ is non-special. In this case, we have

\begin{proposition}[{\cite[{Proposition~6.8,~Corollary~6.11}]{Kuz}}]\label{prop:str} Let $X$ be a smooth linear section $X \subset S_4$ of codimension two. If $X$ is non-special, then we have the following diagram
\[\xymatrix{
E \ar@{^{(}-_{>}}[r]^{} \ar[d]^{}  &  {\rm Bl}_Z(X) \ar[d]^{}  \ar[r]^{} & Q^5  \\
Z \ar@{^{(}-_{>}}[r]^{}& X \\
} \] 
with properties:
\begin{enumerate}
\renewcommand{\labelenumi}{(\roman{enumi})}
\item $Z\cong Q^5 \hookrightarrow X$ is a natural inclusion;
\item ${\rm Bl}_Z(X)$ is the blow-up of $X$ along $Z\cong Q^5$ and $E$ is its exceptional divisor;
\end{enumerate}
\end{proposition}

\begin{proposition}\label{prop:ott:bdl} Under the assumption of Proposition~\ref{prop:str}, $E$ is isomorphic to $\P(\cG_o)$, where $\cG_o$ is an Ottaviani bundle.
\end{proposition}

\begin{proof} Let $X$ be a smooth linear section $X \subset S_4$ of codimension two. By \cite[Lemma~3.3]{Kuz} and the proof of \cite[Proposition~6.8]{Kuz}, we may find a smooth hyperplane section $X_1$ of $S_4 \subset \P^{15}$, a smooth quadric $Z_1=Q^6 \subset X_1$ and a hyperplane $H  \subset \P^{15}$ such that scheme-theoretically $X=X_1\cap H$ and $Z=Z_1\cap H$. Let $E_1$ be the exceptional divisor of the blow-up ${\rm Bl}_{Z_1}(X_1) \to X_1$. According to \cite[Corollary~5.11]{Kuz}, the natural projection $E_1 \to Z_1$ is isomorphic to $q: F \to S_3$ as in the diagram (\ref{eq:diag:2}) in Subsection~\ref{subsec:ott}. Under the identification $E_1 \cong F$, $E$ is $q^{-1}(Q^5)$. Applying Proposition~\ref{prop:ott}, we concludes our assertion.
\end{proof}

\begin{proposition}\label{prop:codim2} Let $X$ be a smooth linear section $X \subset S_4$ of codimension two. If $X$ is non-special, then the tangent bundle of $X$ is not nef.
\end{proposition}

\begin{proof} By \cite[P.98, Table~1]{O}, there exists a special line $\ell$ on $Q^5$ satisfying $\cG|_{\ell} \cong \cO_{\P^1}(2)\oplus \cO_{\P^1}^{\oplus 2}$. Applying Proposition~\ref{prop:ott:bdl}, we have $E=\P(N_{Q^5/X}^{\ast}|_{\ell})\cong \P(\cG_{\ell})$. This tells us that there exists an integer $\alpha$ such that 
$$
N_{Q^5/X}^{\ast}|_{\ell} \cong \cO_{\P^1}(2+\alpha)\oplus \cO_{\P^1}(\alpha)^{\oplus 2}.
$$
Applying the adjunction formula to the triple $\ell \subset Q^5 \subset X$, we see that $\alpha=-1$. This implies that 
$$
N_{Q^5/X}|_{\ell} \cong \cO_{\P^1}(-1)\oplus \cO_{\P^1}(1)^{\oplus 2}.
$$ 
Since we have a surjection $N_{\ell/X} \to N_{Q^5/X}|_{\ell}$, we concludes that $N_{\ell/X}$ is not nef.
\end{proof}

\subsubsection{\rm \bf  The case of $\dim X=7$}

Our goal of this subsection is to prove: 

\begin{proposition}\label{prop:codim3} Let $X$ be a smooth linear section $X \subset S_4$ of codimension three. Then the tangent bundle of $X$ is not nef.
\end{proposition}

We shall prove this by contradiction. 
Let us begin by setting up our notation:
\begin{NA}\label{NA} Let $X$ be a smooth linear section $X \subset S_4$ of codimension three. Assume that the tangent bundle of $X$ is nef. We denote by $\cM_X$ the Hilbert scheme of lines on $X$ and $\cU_X$ its universal family. Then we have the associated universal morphisms $p: \cU_X \rightarrow \cM_X$ and $q: \cU_X \rightarrow X$. For any point $x \in X$, we denote by $\cM_{X, x}$ the Hilbert scheme of lines on $X$ passing through $x$. 
\end{NA}
Under the above setting \ref{NA}, $X \subset \P^{12}$ is a $7$-dimensional CP manifold of $\rho_X=1$. Then we have the following:

\begin{lemma}\label{lem:VMRT} Under the setting as in \ref{NA}, the following holds. 
\begin{enumerate}
\item For any $[{\ell}]\in \cM_X$, we have
$$T_X|_{\ell} \cong \cO_{\P^1}(2) \oplus \cO_{\P^1}(1)^{\oplus 3} \oplus \cO_{\P^1}^{\oplus 3}.$$
\item The evaluation morphism $q$ is a smooth morphism and $p$ is a $\P^1$-bundle. 
\item There exists a closed embedding $\tau: \cU_X \hookrightarrow \P(\Omega_X)$ such that $\tau^*\cO_{\P(\Omega_X)}(1) \cong \omega_{\cU_X/\cM_X}$, where $\cO_{\P(\Omega_X)}(1)$ is the tautological line bundle of $\P(\Omega_X)$ and $\omega_{\cU_X/\cM_X}=\Omega_{\cU_X/\cM_X}$ is the relative canonical line bundle of $p:\cU_X \to \cM_X$, and $\tau$ satisfies the following commutative diagram 
\[\xymatrix{
& \P(\Omega_X) \ar[d]^{} \\
\cU_X \ar[r]^{q} \ar@{^{(}-_{>}}[ur]^{\tau} & X}\]
\item $\cM_X$ is a projective manifold.  
\item The evaluation morphism $q: \cU_X \to X$ is a $V_5$-bundle, where $V_5$ is a smooth linear section of $G(2, 5) \subset \P^9$ of codimension three. 
\end{enumerate}
\end{lemma}

\begin{proof} For any line $[\ell] \in \cM$, $T_X|_{\ell}$ and $T_{\P^{N}}|_{\ell} \cong  \cO_{\P^1}(2) \oplus \cO_{\P^1}(1)^{\oplus 6}$ contain  $T_{\ell}$ and  $T_X|_{\ell}$  as subbundles, respectively. Since $T_X|_{\ell}$ is nef by our assumption, we have 
$$T_X|_{\ell} \cong \cO_{\P^1}(2) \oplus \cO_{\P^1}(1)^{\oplus p} \oplus \cO_{\P^1}^{\oplus 6-p}$$ 
for some integer $p$. Taking the determinant of both sides, we have 
$-K_X\cdot \ell=p+2$. Since the Fano coindex of $X$ is three, $\rm (1)$ holds. 
Furthermore, the former part of $(2)$ follows from \cite[Theorem~2.15 and Corollary~3.5.3]{Kb}. The later part  is trivial. 

Since $q$ is smooth, we have an exact sequence
\begin{eqnarray} q^*\Omega_X \rightarrow \Omega_{\cU_X/\cM_X} \rightarrow 0. \nonumber
\end{eqnarray}
This induces a morphism $\tau: \cU_X \rightarrow \P(\Omega_X)$ over $X$ such that $\tau^*\cO_{\P(\Omega_X)}(1) \cong \Omega_{\cU_X/\cM_X}$. For any $x \in X$, we consider a natural morphism 
$$\tau_x: \cU_{X, x}:=q^{-1}(x) \rightarrow \P(\Omega_{X, x})$$ 
induced from $\tau$. Since $p$ induces an isomorphism 
$$\cU_{X, x} \cong \cM_{X, x}=\{[\ell] \in \cM_X | x \in \ell\},$$
$\tau_x$ can be seen as a morphism $\cM_{X, x} \rightarrow \P(\Omega_{X, x})$, which sends an $\cM_{X, x}$-line to its tangent direction at $x$. Then we see that $\tau_x$ is a closed embedding (see \cite[Proposition~1.5]{Hw} and its proof). Therefore so is $\tau$. Hence $\rm (3)$ holds.

To prove $(4)$, it is enough to prove the irreducibility of $\cM_X$. 
By \cite[II.~Theorem~1.7, Theorem~2.15]{Kb}, $\cM_X$ is an equidimensional disjoint union of projective manifolds. We denote by $\cM_i$ the irreducible components of $\cM_X$ and $\cU_i$ these universal families. The associated universal morphisms are denoted by $p_i: \cU_i \to \cM_i$ and $q_i: \cU_i \to X$. For any point $o \in S_4$, the Hilbert scheme $\cM_{S_4, o} \subset \P(\Omega_{S_4, o})$ of lines passing through $o$ is projectively equivalent to $G(2, 5) \subset \P^9$ by \cite[Theorem~4.3]{LM} (see also \cite[Section~2.4]{MOSWW}). Thus, for any $x \in X$, $\cM_{X, x} \subset \P(\Omega_{X, x})$ is projectively equivalent to $V_5 \subset \P^6$. This implies that $q^{-1}(x) =\bigsqcup\, q_i^{-1}(x)$ is isomorphic to a $3$-dimensional linear section of $G(2, 5) \subset \P^9$. Since a positive-dimensional general linear section of an irreducible variety is again irreducible, any linear section of an irreducible variety is connected. This tells us that $\cM_X$ consists of a single irreducible component, i.e. $\cM_X$ is irreducible. Finally, since all smooth linear sections of $G(2,5) \subset \P^9$ are projectively equivalent by \cite[Theorem~7.6]{Fuj2}, $q^{-1}(x)\subset \P(\Omega_{X, x})$ is projectively equivalent to $V_3$. Thus (5) holds.   
\end{proof}

\begin{theorem}[{\cite[Section~5.1]{KPS}}]\label{them:MU} Let $V_5$ be a $3$-dimensional smooth linear section of $G(2,5) \subset \P^9$. 
Then the automorphism group of $V_5$ is ${\rm Aut}(Y)\cong {\rm PGL(2)}$. Moreover the orbit decomposition of $Y$ is 
$$
Y={O}_3 \sqcup {O}_2 \sqcup {O}_1
$$
with properties
\begin{enumerate}
\item $\dim { O}_i=i$ for any $i$;
\item ${O}_1 \cong {\rm PGL(2)}/B \cong \P^1$, where $B$ is a Borel subgroup of ${\rm PGL(2)}$;  
\item the Hilbert scheme of lines on $V_5$ passing through a point consists of $i$ points for any $1 \leq i \leq 3$.
\end{enumerate}
\end{theorem}

\begin{proof}[Proof of Proposition~\ref{prop:codim3}] The main idea of the following proof is similar to that of \cite{OSWi}. We will freely use the notation appearing in \cite{OSWi}.

Let $X$ be as in \ref{NA}. Under the notation \ref{NA}, we see that the evaluation morphism $q: \cU_X \to X$ is a $V_5$-bundle by Lemma~\ref{lem:VMRT}. Then $q: \cU_X \to X$ is locally trivial in the analytic topology. Thus $q$ defines a cocycle $$\theta \in H^1(X, {\rm Aut} (V_5)) \cong H^1(X,  {\rm PGL(2)}) .$$
On the other hand, the cocycle $\theta \in H^1(X,  {\rm PGL(2)})$ defines a principal ${\rm PGL}(2)$-bundle $E_{{\rm PGL}(2)}$ over $X$. Then, for a Borel subgroup $B$ of ${\rm PGL}(2)$, consider the ${\rm PGL}(2)/B$-bundle ($=$ $\P^1$-bundle) $E_{B}:=E_{{\rm PGL}(2)}\times_{{\rm PGL}(2)}{\rm PGL}(2)/B\to X$. We denote $E_{B}$ by $\cW$. Then $\cW$ satisfies the following commutative diagram
\[\xymatrix{
\cW \ar[dr]\ar@{_{(}-{>}}[d]  &  \\
\cU_X \ar[r]^{q} & X}\] 
We claim that $p: \cW \to p(\cW)$ is a $\P^1$-bundle. To prove this, it is enough to prove that a fiber of $p: \cU_X \to \cM$ is either contained in $\cW$ or disjoint from it. By \cite[Proposition~3.1]{HH} or \cite[Proposition~2.2]{Mk3}, it suffices to show that the second fundamental form of $V_5$ at a point on the minimal closed orbit $O_1$ is not isomorphic to the one at a point on $O_3 \sqcup O_2$.  Since $V_5$ is a linear section of $G(2,5) \subset \P^9$, it is defined by quadrics. Thus the base locus of the second fundamental form of $V_5$ at a point $z$ is exactly the set of lines on $V_5$ through $z$ (see for instance \cite[Corollary~2.3.6]{russo}). Thus, applying Theorem~\ref{them:MU}, the the second fundamental form at $z \in V_5$ is different depending on whether $z$ is on $\cW$ or not. As a consequence, we see that $p: \cW \to p(\cW)$ is a $\P^1$-bundle. Thus $\cW$ admits two different $\P^1$-bundle structures. It follows from \cite[Appendix~A, Theorem~A.1]{OSWi} that $\cW$ is a complete flag manifold $G/B$, where $G$ is a semisimple linear algebraic group of type $A_1 \times A_1, A_2, B_2$ or $G_2$ and $B$ is its Borel subgroup. It turns out that $X$ is isomorphic to $\P^1, \P^2, \P^3, Q^3, Q^5$ or $K(G_2)$. This is a contradiction. 
\end{proof}

\begin{remark} In the above proof, we use the second fundamental form of $V_5$ to prove $p: \cW \to p(\cW)$ is a $\P^1$-bundle. This idea is based on a referee comments on the paper \cite{OSW}. The author would like to thank the anonymous referee  of the paper  \cite{OSW} for valuable comments.
\end{remark}

\bibliographystyle{plain}
\bibliography{biblio}

\end{document}